\documentclass[12pt,reqno]{amsart}
\usepackage{amsmath,amssymb,amscd,latexsym,amsthm,mathrsfs}
\usepackage[unicode]{hyperref}
\usepackage{hypbmsec}
\ifx\pdfoutput\undefined\else
\hypersetup{
colorlinks=true,
linkcolor=blue,
citecolor=blue,
urlcolor=blue,
filecolor=blue,
bookmarksnumbered=true,
pdfstartview=FitH,
pdfhighlight=/N
}
\fi

\newtheorem{theorem}{Theorem}

\theoremstyle{remark}

\def\m{{\operatorname{m}}}
\def\d{{\operatorname{d}}}
\def\Li{{\operatorname{Li}}}

\begin{document}

\hypersetup{pdfauthor={Jes\'us Guillera, Mathew Rogers},%
pdftitle={Mahler measure and the WZ algorithm}}

\title{Mahler measure and the WZ algorithm}

\author{Jes\'us Guillera}
\address{Av.\ Ces\'areo Alierta, 31 esc.~izda 4$^\circ$--A, Zaragoza, SPAIN}
\email{jguillera@gmail.com}

\author{Mathew Rogers}
\address{Department of Mathematics, Universit\'{e} de Montr\'{e}al, Montreal, QC, Canada}
\email{mathewrogers@gmail.com}

\date{May 8, 2013}

\begin{abstract}
We use the Wilf-Zeilberger method to prove identities
 between Mahler measures of polynomials.  In particular, we offer a new proof of
 a formula due to Lal\'{i}n, and we show how to translate the identity into a formula involving elliptic dilogarithms.  This work settles a challenge problem proposed by Kontsevich and Zagier in their paper ``Periods" \cite{KZ}.
\end{abstract}

\maketitle

\section{Introduction}

In this paper we use the Wilf-Zeilberger algorithm to prove relations between Mahler measures of polynomials.
The (logarithmic) Mahler measure of an $n$-variable Laurent
polynomial, $P(x_1,\dots,x_n)$, is defined by
\begin{equation*}
\m(P):=\int_{0}^{1}\dots\int_{0}^{1}\log\left|P\left(e^{2\pi i
\theta_1},\dots,e^{2\pi i
\theta_n}\right)\right|\d\theta_1\dots\d\theta_n.
\end{equation*}
We are primarily interested in the following special
function:
\begin{equation*}
m(\alpha):=\m\left(\alpha+x+\frac{1}{x}+y+\frac{1}{y}\right),
\end{equation*}
because there are many conjectural formulas relating special
values of $m(\alpha)$ to values of $L$-functions attached to elliptic
curves. Deninger hypothesized that $m(1)$
should be a rational multiple of $L(E_{15},2)/\pi^2$, where $E_{15}$ is a
conductor $15$ elliptic curve \cite{De}. Boyd used numerical
calculations to make the constant explicit \cite{Bo1}:
\begin{equation}\label{deninger formula}
m\left(1\right)=\frac{15}{4\pi^2}L\left(E_{15},2\right).
\end{equation}
The second author and Zudilin proved formula \eqref{deninger formula} quite recently \cite{RZ}.  We note that
Boyd's paper contains dozens of additional formulas for $m(\alpha)$, and most of those remain open.

It is usually much easier to prove identities between
Mahler measures, than to prove formulas relating them to
$L$-functions.   One important intermediate step in the proof of \eqref{deninger formula},
is to show that
\begin{equation}\label{matilde's formula}
11m\left(1\right)=m\left(16\right).
\end{equation}
Formula \eqref{matilde's formula} has been the subject of several questions and papers.  The first proof of \eqref{matilde's formula} is due to Lal\'{i}n \cite{La}.  She used the fact that Mahler measures can be interpreted as values of regulator maps on $K_2$ groups of elliptic curves.  The connection to algebraic $K$-theory was outlined and exploited by Rodriguez-Villegas in \cite{RV2}.  An equivalent version of formula \eqref{matilde's formula} appears in the paper ``Periods" of Kontsevich and Zagier \cite{KZ}.  They asked if the relation
\begin{equation}\label{zagier question}
6m(1)=m(5)
\end{equation}
can be proved with elementary calculus \cite[pg.~9]{KZ}, \cite[pg.~56]{Zg}.  Mahler measures are examples of \textit{periods} - numbers which can be expressed as multiple integrals of algebraic functions, over domains described by algebraic equations.  Kontsevich and Zagier conjectured that any relation between periods should be provable with only ``the rules of calculus".  They suggested finding an elementary proof of \eqref{zagier question} as a challenge problem.  The equivalence of \eqref{matilde's formula} and \eqref{zagier question}, follows easily from a result of Kurokawa and Ochiai \cite{KO}:
\begin{equation}
m(1)+m(16)=2m(5).
\end{equation}
In the first section of this paper we present an elementary proof of \eqref{matilde's formula}, answering the Kontsevich-Zagier challenge.  The most difficult part of the proof is to establish formula \eqref{log(2)
formula 3 generalizaion} below, and we accomplish this using the
Wilf-Zeilberger method.

In the second portion of the paper, we present several new
$q$-series expansions for Mahler measures.  If $\varphi(q)$ denotes the standard theta function
\begin{equation}\label{theta def}
\varphi(q):=\sum_{n\in\mathbb{Z}}q^{n^2},
\end{equation}
then
\begin{equation}
m\left(4\frac{\varphi^2(q)}{\varphi^2(-q)}\right)=\frac{4}{\pi}\sum_{n\in\mathbb{Z}}D(i
q^n),
\end{equation}
where $D(z)$ is the Bloch-Wigner dilogarithm.  These results provide an easy way to translate between Mahler measures and
elliptic dilogarithms.  In Theorem \ref{theorem bertin} we
translate an exotic relation due to Bertin into an identity
between Mahler measures \cite{bert2}.  We conclude by briefly comparing our new results to
Ramanujan's formulas for $1/\pi$.

%%%%%%%%%%%%%%%%%%%%%%%%%%%%%%%%%%%%%%%%%%%%%%%%%%%%%%%%%%%%%%%%%

\section{An application of the WZ method}\label{section: WZ}

%%%%%%%%%%%%%%%%%%%%%%%%%%%%%%%%%%%%%%%%%%%%%%%%%%%%%%%%%%%%%%%%%

We begin with a brief review of the WZ method. We say
that $F(n,k)$ is hypergeometric, if $F(n+1,k)/F(n,k)$ and
$F(n,k+1)/F(n,k)$ are rational functions of $n$ and $k$.  Two
hypergeometric functions are called a WZ-pair if they satisfy the
following functional equation:
\begin{equation}\label{WZ functional equation}
F(n+1,k)-F(n,k)=G(n,k+1)-G(n,k).
\end{equation}
Wilf and Zeilberger proved that if $F(n,k)$ satisfies \eqref{WZ
functional equation}, then it is always possible to determine
$G(n,k)$ (see \cite{WZ} and \cite{WZ2}).  Their algorithm has been
implemented in \texttt{Maple} and \texttt{Mathematica}.

Let us consider WZ-pairs where $F$ and $G$ are meromorphic
functions of $n$ and $k$. If we sum both sides of \eqref{WZ
functional equation} from $n=0$ to $n=\infty$, the left-hand side
of the equation telescopes, and we have
\begin{equation*}
-F(0,k)+\lim_{n\rightarrow\infty}F(n,k)=\sum_{n=0}^{\infty}G(n,k+1)-\sum_{n=0}^{\infty}G(n,k).
\end{equation*}
In instances where $F(0,k)=0$, and
$\lim_{n\rightarrow\infty}F(n,k)=0$, this becomes
\begin{equation*}
\sum_{n=0}^{\infty}G(n,k)=\sum_{n=0}^{\infty}G(n,k+1).
\end{equation*}
It follows immediately that the series is periodic with respect to
$k$. If the series also converges uniformly, and $j$ is an
integer, then we can write
\begin{equation*}
\sum_{n=0}^{\infty}G(n,k)=\sum_{n=0}^{\infty}\lim_{j\rightarrow\infty}G(n,k+j).
\end{equation*}
If $\lim_{j\rightarrow\infty}G(n,k+j)$ is independent of $k$, then
we can conclude that for unrestricted $k$:
\begin{equation}\label{Carlson identity}
\sum_{n=0}^{\infty}G(n,k)=\text{constant}.
\end{equation}
We use this method to prove Theorem \ref{main wz theorem} below.  Theorem \ref{main wz theorem} is the key result we need to establish the relation between $m(1)$ and $m(16)$.

In order to apply the WZ-method to equation
\eqref{matilde's formula}, we need to relate the Mahler measures to
hypergeometric functions.  We use several of the identities
summarized in \cite{Rgsubmit}. If $r\in(0,1]$, results from
\cite{KO} and \cite{RV} show that:
\begin{align}
m\left(\frac{4}{r}\right)=&\log\left(\frac{4}{r}\right)-\sum_{n=1}^{\infty}{2n\choose
n}^2\frac{\left(r/4\right)^{2n}}{2n},\label{m(a) rodriguez
expansion}\\
m\left(4r\right)=&4\sum_{n=0}^{\infty}{2n\choose
n}^2\frac{\left(r/4\right)^{2n+1}}{2n+1}.\label{m(a) kurokawa
expansion}
\end{align}
Both of these sums depend upon the same binomial
coefficients. Therefore, if we define $s$ by
\begin{equation*}
s:=\frac{m\left(4/r\right)}{m\left(4r\right)},
\end{equation*}
we can form a linear combination of \eqref{m(a) rodriguez
expansion} and \eqref{m(a) kurokawa expansion} to obtain
\begin{equation}\label{formula for log(4/r)}
\log\left(\frac{4}{r}\right)=r
s+\sum_{n=1}^{\infty}\frac{\left(2(1+r
s)n+1\right)}{(2n)(2n+1)}{2n\choose
n}^2\left(\frac{r}{4}\right)^{2n}.
\end{equation}
It follows that $(r,s)\in\mathbb{\bar{Q}}^2$ and $r\in(0,1]$,
\textit{if and only if} \eqref{formula for log(4/r)} also gives an
explicit formula for an algebraic hypergeometric series.  By
linearity, explicit cases of \eqref{formula for log(4/r)}
immediately imply formulas for $s$.  Notice that \eqref{formula
for log(4/r)} diverges when $|r|>1$.

\begin{theorem}\label{first wz theorem} The following formulas are true:
\begin{align}
2\log(2)=&1+\sum_{n=1}^{\infty}\frac{(4n+1)}{(2n)(2n+1)}{2n\choose
n}^2\frac{1}{2^{4n}},\label{log(2) formula 1}\\
3\log(2)=&2+\sum_{n=1}^{\infty}\frac{(6n+1)}{(2n)(2n+1)}{2n\choose
n}^2\frac{1}{2^{6n}},\label{log(2) formula 2}\\
8\log(2)=&\frac{11}{2}+\sum_{n=1}^{\infty}\frac{(15n+2)}{(2n)(2n+1)}{2n\choose
n}^2\frac{1}{2^{8n}}.\label{log(2) formula 3}
\end{align}
Furthermore, \eqref{log(2) formula 2} is equivalent to
\begin{equation}\label{m(2) and m(8) relation}
4m(2)=m(8),
\end{equation}
and \eqref{log(2) formula 3} is equivalent to
\begin{equation}\label{m(1) and m(16) relation}
11m(1)=m(16).
\end{equation}
\end{theorem}

So far we have not been able to prove equation \eqref{m(2)
and m(8) relation} with the WZ method.
This is surprising, because the $K$-theoretic proof of the relation between $m(2)$ and $m(8)$ \cite{LR}, is much easier than the $K$-theoretic proof of the relation between $m(1)$ and $m(16)$.  In order to prove \eqref{m(1)
and m(16) relation}, we must first prove equation \eqref{log(2) formula 3
generalizaion} below. It is interesting to note that
\texttt{Mathematica} can recognize \eqref{log(2) formula 1}, but not
\eqref{log(2) formula 2} or \eqref{log(2) formula 3}. While it is possible to derive \eqref{log(2)
formula 1} from Dougall's theorem \cite{We}, it seems that equations \eqref{log(2) formula 2} and \eqref{log(2) formula 3} are not as easily accessible.

\begin{theorem}\label{main wz theorem} The following identities are true:
\begin{align}
\pi\frac{\Gamma(x)\Gamma(x+1)}{\Gamma^2\left(x+\frac{1}{2}\right)}=&\sum_{n=0}^{\infty}\frac{(4n+2x+1)}{(2n+1)(n+x)}\frac{\left(\frac{1}{2}+x\right)_n}{(1+x)_n}{2n\choose
n}\frac{1}{2^{2n}},\label{log(2) formula 1 generalization}\\
%???=&\sum_{n=0}^{\infty}???\\
4\pi\frac{\Gamma(x)\Gamma(x+1)}{\Gamma^2\left(x+\frac{1}{2}\right)}
=&\sum_{n=0}^{\infty} \frac{(2(2n+1)^2(15n+2)+x
P(n,x))}{(2n+1)(2n+x)(2n+x+1)^2} \frac{
\left(\frac{1}{2}+x\right)_n^2}{\left(1+\frac{x}{2}\right)_n\left(\frac{1+x}{2}\right)_n}{2n\choose
n}\frac{1}{2^{6n}},\label{log(2) formula 3 generalizaion}
\end{align}
where $P(n,x)=(2n+1)(86n+19)+4x(20n+7)+12x^2$.  The Pochhammer symbol is given by $(x)_m:=\Gamma(x+m)/\Gamma(x)$.
\end{theorem}
\begin{proof}  We begin by proving \eqref{log(2)
formula 1 generalization}. Consider the following
WZ-pair:
\begin{equation}\label{wz pair number one}
\begin{split}
F(n,k)&=-\frac{\left(\frac{1}{2}+k\right)_n\left(\frac{1}{2}\right)_n\left(\frac{1}{2}\right)^2_k}
{\left(1+k\right)_n\left(1\right)_n\left(1\right)^2_k}\cdot\frac{n}{2(n+k)},\\
G(n,k)&=\frac{\left(\frac{1}{2}+k\right)_n\left(\frac{1}{2}\right)_n\left(\frac{1}{2}\right)^2_k}
{\left(1+k\right)_n\left(1\right)_n\left(1\right)^2_k}\cdot\frac{k(4n+2k+1)}{2(n+k)(2n+1)}.
\end{split}
\end{equation}
It is easy to see that $F(0,k)=0$, and
$\lim_{n\rightarrow\infty}F(n,k)=0$.  Since
$\sum_{n=0}^{\infty}G(n,k)$ converges uniformly, we conclude from the previous discussion that
\begin{equation*}
\begin{split}
\sum_{n=0}^{\infty}G(n,k)=&\sum_{n=0}^{\infty}\lim_{j\rightarrow\infty}G(n,k+j)\\
=&\frac{1}{\pi}\sum_{n=0}^{\infty}\frac{1}{(2n+1)}{2n\choose
n}\frac{1}{2^{2n}}=\frac{\arcsin(1)}{\pi}=\frac{1}{2}.
\end{split}
\end{equation*}
Rearrange the formula, and let $k\rightarrow x$ to
complete the proof of \eqref{log(2) formula 1 generalization}.

The proof of \eqref{log(2) formula 3 generalizaion} is similar. Consider the following
WZ-pair:
\begin{equation}\label{wz pair number three}
\begin{split}
F(n,k)=&-U(n,k)\cdot\frac{4n}{2n+k},\\
G(n,k)=&U(n,k)\cdot\frac{2(15n+2)(2n+1)^2+k
P(n,k)}{(2n+k+1)^2(2n+k)(2n+1)}\cdot\frac{k}{2},
\end{split}
\end{equation}
where
\begin{align*}
U(n,k)=\frac{1}{16^n}\frac{\left(\frac{1}{2}+k\right)_n^2\left(\frac{1}{2}\right)_n}
{\left(1+\frac{k}{2}\right)_n\left(\frac{1}{2}+\frac{k}{2}\right)_n\left(1\right)_n}\frac{\left(\frac{1}{2}\right)_k^2}{(1)_k^2},
\end{align*}
and $P(n,k)$ is given in the statement of the theorem.  Then $F(0,k)=0$, $\lim_{n\rightarrow\infty}F(n,k)=0$, and
$\sum_{n=0}^{\infty}G(n,k)$ converges uniformly, so we have
\begin{equation*}
\begin{split}
\sum_{n=0}^{\infty}G(n,k)=&\sum_{n=0}^{\infty}\lim_{j\rightarrow\infty}G(n,k+j)\\
=&\frac{6}{\pi}\sum_{n=0}^{\infty}\frac{1}{(2n+1)}{2n\choose
n}\frac{1}{2^{4n}}=\frac{12\arcsin(1/2)}{\pi}=2.
\end{split}
\end{equation*}
Rearranging the final result, and relabeling $k$ as $x$ completes
the proof of \eqref{log(2) formula 3 generalizaion}.
%$\blacksquare$
\end{proof}
\begin{proof}[\textit{Proof of Theorem \ref{first wz
theorem}.}]

The shortest proof of \eqref{log(2) formula 1} follows from using
the definition of $s$, to show that $s=1$ when $r=1$.  An
alternative proof follows from using \eqref{log(2) formula 1
generalization} to show that:
\begin{equation*}
\begin{split}
 2+2\sum_{n=1}^{\infty}\frac{(4n+1)}{(2n)(2n+1)}{2n\choose
n}^2\frac{1}{2^{4n}}&=\lim_{x\rightarrow
0}\left(\pi\frac{\Gamma(x)\Gamma(x+1)}{\Gamma^2(x+1/2)}-\frac{1}{x}\right)\\
&=4\log(2).
\end{split}
\end{equation*}
Similarly, \eqref{log(2) formula 3} follows from using
\eqref{log(2) formula 3 generalizaion}, to show that
\begin{equation*}
\begin{split}
 11+2\sum_{n=1}^{\infty}\frac{(15n+2)}{(2n)(2n+1)}{2n\choose
n}^2\frac{1}{2^{8n}}&=\lim_{x\rightarrow
0}\left(4\pi\frac{\Gamma(x)\Gamma(x+1)}{\Gamma^2(x+1/2)}-\frac{4}{x}\right)\\
&=16\log(2).%\quad%\blacksquare
\end{split}
\end{equation*}
\end{proof}

We can also use our newly-discovered WZ pairs to obtain some formulas for $\zeta(3)$.

\begin{theorem}The following identities are true:
\begin{align}
\zeta(3)&=\frac{2}{7}\sum_{n=0}^{\infty}\frac{(4n+3)16^n}{(2n+1)^3(n+1){2n\choose
n}^2}\label{zeta(3) formula 3},\\
\zeta(3)&\stackrel{?}{=}\frac{4}{7}\sum_{n=0}^{\infty}\frac{(3n+2)4^n}{(2n+1)^3(n+1){2n\choose
n}^2}\label{zeta(3) formula 2},\\
\zeta(3)&=\frac{1}{16}\sum_{n=0}^{\infty}\frac{(30n+19)}{(2n+1)^3(n+1){2n\choose
n}^2}\label{zeta(3) formula 1}.
\end{align}
\end{theorem}
\begin{proof}  The idea is that shifting the summation in formulas \eqref{log(2) formula 1}, \eqref{log(2) formula 2},
and \eqref{log(2) formula 3} by $n\rightarrow n-1/2$, changes
them into formulas for $\zeta(3)$ \cite{Gu2}. The
summand in \eqref{log(2) formula 2} becomes
\begin{equation*}
\frac{(6n+1)}{(2n)(2n+1)}{2n\choose
n}^2\frac{1}{2^{6n}}\rightarrow\frac{2(3n+2)4^n}{\pi^2(2n+1)^3(n+1){2n\choose
n}^2}.
\end{equation*}
To make this observation rigorous, we prove
\eqref{zeta(3) formula 1} and \eqref{zeta(3) formula 3} by using
the WZ-pairs from Theorem \ref{main wz theorem}.  A
rigorous proof of \eqref{zeta(3) formula 2} should be easy to
construct, by first finding a WZ proof of \eqref{log(2) formula 2}.

Shift both sides of \eqref{WZ functional equation} by
$\frac{1}{2}$ and $y$, and then sum the equation from $n=0$ to
$n=\infty$.  Under the hypothesis that
$\lim_{n\rightarrow\infty}F\left(n+\frac{1}{2},k+y\right)=0$, we
have
\begin{equation*}
F\left(\frac{1}{2},k+y\right)=-\sum_{n=0}^{\infty}G\left(n+\frac{1}{2},k+y+1\right)+\sum_{n=0}^{\infty}G\left(n+\frac{1}{2},k+y\right).
\end{equation*}
The right-hand side of the formula telescopes with respect to $k$.
Sum both sides of the equation from $k=0$ to
$k=\infty$:
\begin{equation*}
\sum_{k=0}^{\infty}F\left(\frac{1}{2},k+y\right)=-\lim_{k\rightarrow\infty}\sum_{n=0}^{\infty}G\left(n+\frac{1}{2},k+y\right)
+\sum_{n=0}^{\infty}G\left(n+\frac{1}{2},y\right).
\end{equation*}
In the cases we consider, all three sums converge uniformly,
the limit are independent of $y$, and
$G\left(n+\frac{1}{2},0\right)=0$.  Differentiating with respect
to $y$, and using the notation
$F^{*}(n,k)=\frac{\partial}{\partial k}F(n,k)$, and
$G^{*}(n,k)=\frac{\partial}{\partial k}G(n,k)$, reduces the formula
to
\begin{equation}\label{zeta(3) F* G* sum}
\sum_{k=0}^{\infty}F^{*}\left(\frac{1}{2},k\right)=\sum_{n=0}^{\infty}G^{*}\left(n+\frac{1}{2},0\right).
\end{equation}
Notice that $G^{*}\left(n+\frac{1}{2},0\right)=\lim_{y\rightarrow
0}\frac{G\left(n+\frac{1}{2},y\right)}{y}$.

 If we use $F$ and $G$
from \eqref{wz pair number one}, then we obtain
\begin{equation*}
\begin{split}
\sum_{n=0}^{\infty}G^{*}\left(n+\frac{1}{2},0\right)&=\frac{2}{\pi^2}\sum_{n=0}^{\infty}\frac{(4n+3)16^{n}}{(2n+1)^3(n+1){2n\choose
n}^2}.
\end{split}
\end{equation*}
The sum involving $F^{*}$ is easy to evaluate. Notice
that
\begin{equation*}
F\left(\frac{1}{2},k\right)=-\frac{2}{\pi^2(2k+1)^2},
\end{equation*}
and therefore
\begin{equation*}
\sum_{k=0}^{\infty}F^{*}\left(\frac{1}{2},k\right)=\frac{8}{\pi^2}\sum_{k=0}^{\infty}\frac{1}{(2k+1)^3}=\frac{7\zeta(3)}{\pi^2}.
\end{equation*}
Formula \eqref{zeta(3) formula 3} follows from substituting these
results into \eqref{zeta(3) F* G* sum}.

In order to prove \eqref{zeta(3) formula 1}, we use the
WZ-pair given in \eqref{wz pair number three}.  Observe
that
\begin{equation*}
\sum_{n=0}^{\infty}G^{*}\left(n+\frac{1}{2},0\right)=\frac{1}{4\pi^2}\sum_{n=0}^{\infty}\frac{(30n+19)}{(2n+1)^3(n+1){2n\choose
n}^2},
\end{equation*}
which matches the right-hand side of \eqref{zeta(3) formula 1} up
to a constant.  To evaluate the $F^{*}$-sum, notice that
\begin{equation*}
F\left(\frac{1}{2},k\right)=\frac{-2}{\pi^2(k+1)^2},
\end{equation*}
and therefore
\begin{equation*}
\sum_{k=0}^{\infty}F^{*}\left(\frac{1}{2},k\right)=\frac{4}{\pi^2}\sum_{k=0}^{\infty}\frac{1}{(k+1)^3}=\frac{4\zeta(3)}{\pi^2}.
\end{equation*}
Substituting the last two results into \eqref{zeta(3) F* G* sum}
completes the proof of \eqref{zeta(3) formula 1}. %$\blacksquare$
\end{proof}

Notice that Gosper first proved equation \eqref{zeta(3) formula 1} \cite{GS}, and Batir proved \eqref{zeta(3) formula 3} using
log-sine integrals (combine formulas 3 and 4 on page 664 of
\cite{Ba}).  Formula \eqref{zeta(3) formula
2} is numerically true, but it remains open.

We conclude this section, by showing that it is also possible
to find WZ-pairs when \eqref{formula for log(4/r)} diverges. If we
consider the WZ-pair:
\begin{equation*}
\begin{split}
F(n,k)=&\frac{n}{(2n+k)^2}\cdot\frac{\left(\frac{1}{2}\right)_n\left(1+\frac{k}{2}\right)_n\left(\frac{1}{2}+\frac{k}{2}\right)_n}{\left(1\right)_n\left(1+k\right)^2_n}
\cdot 16^n,\\
G(n,k)=&-\frac{P(n,k)}{n(2n+k)^2(1+2n+k)}\cdot\frac{\left(\frac{1}{2}\right)_n\left(1+\frac{k}{2}\right)_n\left(\frac{1}{2}+\frac{k}{2}\right)_n}{\left(1\right)_n\left(1+k\right)^2_n}
\cdot 16^n,
\end{split}
\end{equation*}
where $P(n,k)=3k^3+k^2(20n+3)+k n(43n+12)+n^2(30n+11)$, then it is
possible to obtain a finite summation identity
\begin{equation}\label{divergent sum formula}
\begin{split}
\sum_{n=1}^{m-1}\frac{30n+11}{(2n)(2n+1)}{2n\choose
n}^2=&-4+6\sum_{n=1}^{m-1}\frac{1}{n}{2n\choose
n}\\
&+\frac{1}{2m}{2m\choose m}^2{_4F_3}\left(\substack{1,1,2m,2m\\
m+1,m+1,2m+1};1\right),
\end{split}
\end{equation}
which holds for $m\in \mathbb{N}$. This formula corresponds to the values $(r,s)=(4,1/11)$, but notice that equation \eqref{m(1) and
m(16) relation} already shows that $s=1/11$ when $r=4$.

%%%%%%%%%%%%%%%%%%%%%%%%%%%%%%%%%%%%%%%%%%%%%%%%%%%%%%%%%%%%%%%%%

\section{Connections with the elliptic dilogarithm}\label{section: elliptic dilog}

%%%%%%%%%%%%%%%%%%%%%%%%%%%%%%%%%%%%%%%%%%%%%%%%%%%%%%%%%%%%%%%%%

In Section \ref{section: WZ} we used the WZ method to establish a relation between Mahler
measures.  In this section, we show that our techniques provide a new way to
prove relations between elliptic dilogarithms.  Briefly
recall the definitions of $m(\alpha)$ and $n(\alpha)$:
\begin{align*}
m(\alpha):=&\m\left(\alpha+x+x^{-1}+y+y^{-1}\right),\\
n(\alpha):=&\m\left(x^3+y^3+1-\alpha x y\right).
\end{align*}
We examined $m(\alpha)$ in the previous section, and
$n(\alpha)$ has been studied in papers such as \cite{RV},
\cite{Rgsubmit} and \cite{LR}. We begin by proving
new $q$-series expansions for both functions.

\begin{theorem}\label{elliptic dilog theorem} Suppose that $q\in (-1,1)$, and let $D(z)=\Im\left(\Li_2(z)+\log|z|\log(1-z)\right)$
denote the Bloch-Wigner dilogarithm.  The following formulas are
true:
\begin{align}
\frac{4}{\pi}\sum_{n\in\mathbb{Z}}D\left(i
q^n\right)=&m\left(4\frac{\varphi^2(q)}{\varphi^2(-q)}\right)
\label{m q-series expansion},\\
\frac{9}{2\pi}\sum_{n\in\mathbb{Z}}D\left(e^{2\pi i/3}
q^n\right)=&n\left(3\frac{a(q)}{b(q)}\right)
\label{n q-series expansion},\\
\frac{9}{\pi}\sum_{n\in\mathbb{Z}}D\left(e^{\pi i/3}
q^n\right)=&2n\left(3\frac{a(q)}{b(q)}\right)
+n\left(3\frac{a(q^2)}{b(q^2)}\right).\label{n second q-series
expansion}
%\frac{9}{2\pi}\sum_{n=-\infty}^{\infty}D\left(e^{2\pi i/3}
%q^n\right)=&n\left(3\sqrt[3]{x(q)}\right)
%\label{n q-series expansion},\\
%\frac{9}{\pi}\sum_{n=-\infty}^{\infty}D\left(e^{\pi i/3}
%q^n\right)=&2n\left(3\sqrt[3]{x(q)}\right)
%+n\left(3\sqrt[3]{x(q^2)}\right),\label{n second q-series
%expansion}
\end{align}
The signature-$3$ theta functions are given by
\begin{align*}
a(q):=\sum_{(n,m)\in\mathbb{Z}^2}q^{n^2+m n+m^2}, && b(q):=\sum_{(n,m)\in\mathbb{Z}^2}\omega^{n-m} q^{n^2+m n+m^2},
\end{align*}
where $\omega=e^{2\pi i/3}$ \cite{Be5}.
\end{theorem}
\begin{proof}  First notice that \eqref{n q-series expansion} implies \eqref{n second
q-series expansion}.  By elementary functional equations for
the Bloch-Wigner dilogarithm,  $\frac{1}{2}D(z^2)=D(z)+D(-z)$, and
$D(z)=-D(\frac{1}{z})$, it is possible to obtain
\begin{equation*}
D\left(e^{\pi i/3} q^n\right)=D\left(e^{2\pi i/3}
q^{-n}\right)+\frac{1}{2}D\left(e^{2\pi i/3} q^{2n}\right).
\end{equation*}
Summing over $n$ shows that \eqref{n q-series expansion} implies \eqref{n second q-series expansion}.

To prove \eqref{m q-series expansion} and \eqref{n q-series
expansion}, we use an idea described in \cite[Section~8]{Rgsubmit}.  Consider the following formula from
\cite{Rgsumbit2} (Rodriguez-Villegas first proved a version of
this formula in \cite{RV}):
\begin{equation*}
\frac{\pi^2}{32x}m\left(4\sqrt{1-\frac{\varphi^4(-q)}{\varphi^4(q)}}\right)
=\sum_{\substack{n=0\\k=0}}^{\infty}\frac{(-1)^{n}(2n+1)}{\left((2n+1)^2+x^2(2k+1)^2\right)^2},
\end{equation*}
where $q=e^{-\pi x}$, and $x>0$.  Express the sum as an integral,
and then apply the involution for the weight-$1/2$ theta function:
\begin{align*}
=&\frac{\pi^2}{16}\int_{0}^{\infty}u\left(\sum_{k=0}^{\infty}e^{-\pi
(k+1/2)^2 x^2 u}\right)\left(\sum_{n=0}^{\infty}(-1)^n(2n+1)
e^{-\pi
(n+1/2)^2 u}\right)\d u\\
=&\frac{\pi^2}{32x}\int_{0}^{\infty}\sqrt{u}\left(\sum_{k=-\infty}^{\infty}(-1)^k
e^{-\frac{\pi k^2}{x^2
u}}\right)\left(\sum_{n=0}^{\infty}(-1)^n(2n+1) e^{-\pi
(n+1/2)^2 u}\right)\d u\\
=&\frac{\pi}{8x}\sum_{k=-\infty}^{\infty}(-1)^k
\sum_{n=0}^{\infty}\frac{(-1)^n}{(2n+1)}
\left(\frac{\pi|k|}{x}+\frac{1}{(2n+1)}\right)e^{-\pi
(2n+1)|k|/x}\\
=&\frac{\pi}{8x}\sum_{k=-\infty}^{\infty}(-1)^k D\left(i e^{-\pi
|k|/x}\right).
\end{align*}
If we let $x\rightarrow 1/x$, and then use the following identity
\begin{equation*}
\frac{\varphi^4(-q)}{\varphi^4(q)}=1-\frac{\varphi^4(-e^{-\pi/x})}{\varphi^4(e^{-\pi/x})},
\end{equation*}
it is easy to see that
\begin{equation*}
 m\left(4\frac{\varphi^2(-q)}{\varphi^2(q)}\right)
 =\frac{4}{\pi}\sum_{k\in\mathbb{Z}}(-1)^k D\left(i
 q^{|k|}\right)=\frac{4}{\pi}\sum_{k\in\mathbb{Z}}D\left(i
 (-q)^{k}\right).
\end{equation*}
The second step uses $(-1)^k
D\left(i q^{|k|}\right)=D\left(i(-q)^k\right)$. Formula \eqref{m
q-series expansion} then follows from sending $q\rightarrow -q$.
A different proof can be constructed by differentiating
\eqref{m q-series expansion} with respect to $q$, and then
applying the formulas of Ramanujan.

A proof of \eqref{n q-series expansion} can be obtained by looking
at the following sum:
\begin{equation*}
\sum_{(n,k)\in\mathbb{Z}^2}\frac{(3k+1)}{\left((3k+1)^2+x^2
(2n+1)^2\right)^2}.
\end{equation*}
If the series is transformed into an integral of theta
functions, then the involution for the weight-$1/2$ theta function
leads to a dilogarithm sum, and the involution for the
weight-$3/2$ theta function leads to
Rodriguez-Villegas's $q$-series for $n(\alpha)$ (see formula
(2-10) in \cite{LR}). %$\blacksquare$
\end{proof}
For certain values of $q$ the left-hand sides of equations \eqref{m
q-series expansion}, \eqref{n q-series expansion}, and \eqref{n
second q-series expansion} equal elliptic dilogarithms.  To see this, briefly consider an elliptic curve
\begin{equation*}
E: y^2=4x^3-g_2 x-g_3.
\end{equation*}
Then $E$ can be parameterized by
$(x,y)=\left(\wp(u),\wp'(u)\right)$, where $\wp(u)$ is the
Weierstrass function.  The periods of $\wp(u)$ are denoted $\omega$ and $\omega'$, and the period ratio
$\tau=\frac{\omega'}{\omega}$ is assumed to have $\Im(\tau)>0$. If
$P=\left(\wp(u),\wp'(u)\right)$ denotes an arbitrary point on $E$,
and $q=e^{2\pi i \tau}$, then the elliptic dilogarithm is defined
by
\begin{equation*}
D^{E}(P):=\sum_{n\in\mathbb{Z}}D\left(e^{2\pi i u/\omega}
q^{n}\right).
\end{equation*}
Since we will only be interested in torsion points, we can assume
that $u=a \omega+b\omega'$, for some $(a,b)\in\mathbb{Q}^2$.  For
appropriate choices of $E$, the series expansions in Theorem
\ref{elliptic dilog theorem} equal $D^{E}(P)$ at $3$, $4$ and
$6$-torsion points.

Now we focus on equation \eqref{m q-series expansion}.  In order to translate the right-hand side into elliptic dilogarithms, we must identify the relevant elliptic elliptic curves and torsion points. Let us set
$\beta=1-\varphi^4(-q)/\varphi^4(q)$. The classical theory
of elliptic functions shows that we can calculate $q$ as a
function of $\beta$:
\begin{equation}\label{q-inversion}
q=e^{2\pi i\tau}
=\exp\left(-\pi\frac{{_2F_1}\left(\substack{\frac{1}{2},\frac{1}{2}\\
1};1-\beta\right)}{{_2F_1}\left(\substack{\frac{1}{2},\frac{1}{2}\\
1};\beta\right)}\right).
\end{equation}
It is known that $g_2$ and $g_3$ are also functions of $q$
\cite{We2}. In Ramanjujan's notation we have $g_2=\frac{4}{3}\pi^4
M(q)$, and $g_3=\frac{8}{27}\pi^6 N(q)$ \cite[pg. 126]{Be3}.
Applying formulas (13.3) and (13.4) in \cite[pg. 127]{Be3}, we
obtain:
\begin{equation}\label{J function beta calculation}
J(\tau)=\frac{g_2^3}{g_2^3-27g_3^2}%=\frac{4\left(1-\alpha+\alpha^2\right)^3}{27\alpha^2(1-\alpha)^2}
=\frac{\left(1+14\beta+\beta^2\right)^3}{108\beta(1-\beta)^4}.
\end{equation}
This relation allows us to easily translate between $\beta$ and $(g_2,g_3)$.  We have six choices of $\beta$ for every
$g_2^3/g_3^2$, so caution must be exercised to pick the correct
$\beta$. We often checked our work numerically by
computing $q$ from $g_2$ and $g_3$ (with the \texttt{Mathematica} function
``WeierstrassHalfPeriods"), and then comparing it to calculations
using \eqref{q-inversion}.

\begin{theorem}\label{theorem : elliptic dilog equivalent} Let $E(k,\ell)$ denote the elliptic
curve:
\begin{equation*}
y^2=4x^3-27(k^4-16k^2+16)\ell^2
x+27\left(k^6-24k^4+120k^2+64\right)\ell^3.
\end{equation*}
Formula \eqref{m(1) and m(16) relation} is equivalent to
\begin{equation}\label{m(1) and m(16) elliptic dilog equivalency}
11 D^{E_1}(P_1)=6 D^{E_2}(P_2),
\end{equation}
where $E_1=E(5,2)$, $E_2=E(16,1/2)$, $P_1=(87,1080)$, and
$P_2=(195,432)$.

Formula \eqref{m(2) and m(8) relation} is equivalent to
\begin{equation}\label{m(2) and m(8) elliptic dilog equivalency}
5 D^{E_3}(P_3)=8 D^{E_4}(P_4),
\end{equation}
where $E_3=E(8,1/2)$, $E_4=E(3\sqrt{2},1)$, $P_3=(51,216)$, and
$P_4=(33,324)$.
\end{theorem}
\begin{proof} If we set $\beta=1-16/k^2$, %Notice that if $k>4$, then $\beta\in(0,1)$.
then we can rearrange \eqref{J function beta calculation} to
obtain
\begin{equation*}
\frac{g_2^3}{g_3^2}=\frac{27\left(16-16k^2+k^4\right)^3}{\left(64+120k^2-24k^4+k^6\right)^2}.
\end{equation*}
Therefore, for some choice of $\ell$, we have
\begin{align*}
g_2&=27\left(k^4-16k^2+16\right)\ell^2,\\
g_3&=-27\left(k^6-24k^4+120k^2+64\right)\ell^3.
\end{align*}
In practice, we choose $\ell$ so that $E(k,\ell)$ has a rational
$4$-torsion point $P$.  We can then use
$4\varphi^2(q)/\varphi^2(-q)=k$, along with equation \eqref{m
q-series expansion}, to conclude that
\begin{equation*}
m(k)=D^{E(k,\ell)}(P).
\end{equation*}

Now we prove \eqref{m(1) and m(16) elliptic dilog
equivalency}.  It is easy to check that $E(5,2)$ has a
$4$-torsion point
$P_1=(87,1080)=\left(\wp\left(\frac{\omega}{4}\right),\wp'\left(\frac{\omega}{4}\right)\right)$.
%The periods of this elliptic curve are
%$(\omega,\omega')\approx(.46686,.26004i)$, and therefore it can be
%numerically verified that $\wp\left(\frac{\omega}{4}\right)=87$.
It follows from the definition of the elliptic dilogarithm, that
\begin{equation*}
m(5)=\frac{4}{\pi}D^{E(5,2)}(P_1).
\end{equation*}
A result from \cite{LR} shows that $m(1)+m(16)=2m(5)$, and
therefore we have proved that
\begin{equation}\label{m(1) and m(16) elliptic dilog relation}
m(1)+m(16)=\frac{8}{\pi}D^{E(5,2)}(P_1).
\end{equation}
When $k=16$, it is easy to see that $E(16,1/2)$ has the
$4$-torsion point
$P_2=(195,432)=\left(\wp\left(\frac{\omega}{4}\right),\wp'\left(\frac{\omega}{4}\right)\right)$.
%Using Mathematica we find that $(\omega,
%\omega')\approx(.46686,.13302i)$, and as expected
%$\wp\left(\frac{\omega}{4}\right)=195$.
Using the definition of the elliptic dilogarithm, we conclude that
\begin{equation}\label{m(16) elliptic dilog relation}
m(16)=\frac{4}{\pi}D^{E(16,1/2)}(P_2).
\end{equation}
Substituting \eqref{m(1) and m(16) elliptic dilog relation} and
\eqref{m(16) elliptic dilog relation} into \eqref{m(1) and m(16)
relation} completes the proof of \eqref{m(1) and m(16) elliptic
dilog equivalency}.

The proof of \eqref{m(2) and m(8) elliptic dilog
equivalency} is more or less identical, except we require the relation $m(2)+m(8)=2m(3\sqrt{2})$.
\end{proof}

The key point of Theorem \ref{theorem : elliptic dilog equivalent}, is that we can start from elliptic dilogarithm identities such as \eqref{m(1) and m(16) elliptic dilog equivalency} or \eqref{m(2) and m(8) elliptic dilog equivalency}, translate both sides into hypergeometric functions, and then prove the hypergeometric identities with the WZ method.  It seems plausible that some additional formulas involving elliptic
dilogarithms might be provable with the WZ method.  Bloch and Grayson conjectured several identities of the form:
\begin{equation*}
\sum_{r}a_r D^{E}(rP)\stackrel{?}{=}0,
\end{equation*}
which they refer to as ``exotic relations" \cite{BG}.
Zagier later proposed restrictions that $E$ should satisfy
in order to possess such a relation \cite{GL}. Theorem
\ref{elliptic dilog theorem} implies that certain exotic relations
are equivalent to formulas for hypergeometric functions.  We
conclude this section by translating an exotic relation due to
Bertin into an identity between Mahler measures \cite{bert2}.

\begin{theorem}\label{theorem bertin} Let $E$ denote the elliptic curve
\begin{equation*}
y^2=4x^3-432 x+1188,
\end{equation*}
and let $P=(-6,54)$.  Bertin's exotic relation
\begin{equation}\label{bertin exotic}
16D^{E}(P)-11D^{E}(2P)=0,
\end{equation}
is equivalent to
\begin{equation}\label{bertin equivalency}
16n\left(\frac{7+\sqrt{5}}{\sqrt[3]{4}}\right)-8n\left(\frac{7-\sqrt{5}}{\sqrt[3]{4}}\right)=19n\left(\sqrt[3]{32}\right).
\end{equation}
\end{theorem}
\begin{proof}
If we notice that
$P=(-6,54)=\left(\wp\left(\frac{\omega-3\omega'}{6}\right),\wp'\left(\frac{\omega-3\omega'}{6}\right)\right)$,
then \eqref{bertin exotic} is equivalent to
\begin{equation*}
16\sum_{n\in\mathbb{Z}}D\left(e^{\pi i/3}
q^{n-1/2}\right)=11\sum_{n\in\mathbb{Z}}D\left(e^{2\pi i/3}
q^n\right).
\end{equation*}
By formulas \eqref{n q-series expansion} and \eqref{n second
q-series expansion}, this amounts to showing that
\begin{equation*}
0=16n\left(3\frac{a(\sqrt{q})}{b(\sqrt{q})}\right)-19n\left(3\frac{a(q)}{b(q)}\right)-8n\left(3\frac{a(q^2)}{b(q^2)}\right).
\end{equation*}
We translate the ratios of theta functions into algebraic numbers below.
%\begin{equation*}
%0=16n\left(3\sqrt[3]{x\left(\sqrt{q}\right)}\right)-19n\left(3\sqrt[3]{x(q)}\right)-8n\left(3\sqrt[3]{x(q^2)}\right).
%\end{equation*}

It is well known that the following inverse relation holds \cite{Borwein}:
\begin{align*}
\beta=\frac{c^3(q)}{a^3(q)},&&
q=
\exp\left(-\frac{2\pi}{\sqrt{3}}\frac{{_2F_1}\left(\substack{\frac{1}{3},\frac{2}{3}\\
1};1-\beta\right)}{{_2F_1}\left(\substack{\frac{1}{3},\frac{2}{3}\\
1};\beta\right)}\right),
\end{align*}
where $a(q)$ and $b(q)$ appear in Theorem \ref{theorem : elliptic dilog equivalent}, and $c(q)$ is given by
\begin{equation*}
c(q):=\sum_{(n,m)\in\mathbb{Z}^2}q^{(n+1/3)^2+(m+1/3)(n+1/3)+(m+1/3)^2}.
\end{equation*}
It is also known that $a^3(q)=b^3(q)+c^3(q)$ \cite{Borwein}.  By formulas (4.6) and (4.8) in \cite[pg. 107]{Be5}, and the
values $g_2=\frac{4}{3}\pi^4 M(q)=432$, and $g_3=\frac{8}{27}\pi^6
N(q)=-1188$, we have
\begin{equation*}
\frac{6912}{6971}=\frac{g_2^3}{g_2^3-27g_3^2}=\frac{(1+8\beta)^3}{64\beta(1-\beta)^3}.
\end{equation*}
Therefore $\beta=\frac{5}{32}$.  Since $\frac{a(q)}{b(q)}=\frac{1}{\sqrt[3]{1-\beta}}$, we have
\begin{equation*}
\frac{a(q)}{b(q)}=\frac{1}{\sqrt[3]{1-\beta}}=\frac{\sqrt[3]{32}}{3}.
\end{equation*}
Finally, if we write $\frac{a^3(\sqrt{q})}{b^3(\sqrt{q})}=\frac{1}{1-\alpha}$ and
$\frac{a^3(q^2)}{b^3(q^2)}=\frac{1}{1-\gamma}$, then $\alpha$ and $\gamma$ are
conjugate zeros of a second-degree modular polynomial with respect
to $\beta$ \cite[pg. 94]{LR}:
\begin{equation*}
27\alpha\beta(1-\alpha)(1-\beta)-(\alpha+\beta-2\alpha\beta)^3=0.
\end{equation*}
With the aid of a computer, we obtain
$\frac{a(\sqrt{q})}{b(\sqrt{q})}=\frac{7+\sqrt{5}}{3\sqrt[3]{4}}$, and
$\frac{a(q^2)}{b(q^2)}=\frac{7-\sqrt{5}}{3\sqrt[3]{4}}$. %$\blacksquare$
\end{proof}

The hypergeometric form of $n(\alpha)$ is due to Rodriguez-Villegas (see
formula (2-36) in \cite{LR}).  If $|\alpha|$ is sufficiently large, then
\begin{equation}\label{n hypergeometric}
n(\alpha)=\Re\left(\log\alpha-\frac{2}{\alpha^3}{_4F_3}\left(\substack{\frac43,\frac53,1,1\\2,2,2};\frac{27}{\alpha^3}\right)\right).
\end{equation}
If follows that \eqref{bertin equivalency} can be rewritten as a series identity.  Unfortunately, it seems doubtful that a WZ proof of \eqref{bertin equivalency} is possible.  The arguments of the hypergeometric functions are irrational numbers, while virtually all of the known WZ proofs deal with rational hypergeometric functions \cite{GZ}.

%%%%%%%%%%%%%%%%%%%%%%%%%%%%%%%%%%%%%%%%%%%%%%%%%%%%%%%%%%%%%%%%%

\section{Conclusion}\label{section: Conclusion}

%%%%%%%%%%%%%%%%%%%%%%%%%%%%%%%%%%%%%%%%%%%%%%%%%%%%%%%%%%%%%%%%%

We conclude by comparing our new results to Ramanujan's formulas
for $1/\pi$.  One of Ramanujan's major insights was to
find formulas such as
\begin{equation}\label{Ramanujan}
\frac{1}{\pi}=\sum_{n=0}^{\infty}(a+b n){2n\choose n}^3 \frac{z^n}{2^{6n}},
\end{equation}
where $a$, $b$, and $z$ are parameterized by logarithmic, modular, and quasi-modular functions \cite{Borwein}, \cite{Chudnovsky}:
\begin{align*}
z=&4\frac{\varphi^4(-q)}{\varphi^4(q)}\left(1-\frac{\varphi^4(-q)}{\varphi^4(q)}\right),\\
a=&\frac{1}{\pi\varphi^4(q)}\left(1+\frac{8\log|q|}{\varphi(q)}\sum_{n=1}^{\infty}n^2 q^{n^2}\right),\\
b=&\frac{\log|q|}{\pi}\left(1-2\frac{\varphi^4(-q)}{\varphi^4(q)}\right).
\end{align*}
We can produce infinitely many irrational algebraic triplets $(a,b,z)$ which make \eqref{Ramanujan} valid.  The parameters are simultaneously algebraic whenever $q=e^{2\pi i \tau}$, with $\tau$ a quadratic irrational in the upper half plane (some additional restrictions on $\tau$ are necessary to ensure that the infinite series in \eqref{Ramanujan} converges).  In fact, it is known that $z$ is algebraic whenever $\tau$ is the period ratio of an elliptic curve, but the values of $a$ and $b$ are only algebraic if the elliptic curve has complex multiplication (and this is less than obvious).

Now consider the fact that we only detected three algebraic pairs $(r,s)$ for which
\begin{equation*}
\log\left(\frac{4}{r}\right)=r
s+\sum_{n=1}^{\infty}\frac{\left(2(1+r
s)n+1\right)}{(2n)(2n+1)}{2n\choose
n}^2\left(\frac{r}{4}\right)^{2n}.
\end{equation*}
We can use formula \eqref{m q-series expansion} to deduce $q$-parameterizations for
$r$ and $s$:
\begin{align}\label{s and q relation}
r=\frac{\varphi^2(-q)}{\varphi^2(q)},&&s=\frac{\mathscr{L}(i,q)}{\mathscr{L}(i,-q)},
\end{align}
where
\begin{align*}
\varphi(q)=\sum_{n\in\mathbb{Z}}q^{n^2},&&\mathscr{L}(i,q)=\sum_{n\in\mathbb{Z}}D(i
q^n).
\end{align*}
It follows that $r$ is a modular function and $s$ is not.  It seems that a small miracle has to occur for $r$ and $s$ to be algebraic simultaneously.  We conjecture that $r$ and $s$ are algebraically independent for almost all values of $q=e^{2\pi i \tau}$, where $\tau$ is the period ratio of an elliptic curve.  Verifying this conjecture will likely require a careful investigation of divisors of elliptic curves.


\bibliographystyle{amsplain}

\end{document}